\documentclass[reqno]{amsart}
\usepackage[colorlinks=true]{hyperref} 
\usepackage{mathrsfs}
\usepackage[numbers,compress]{natbib}
\usepackage{amssymb,times}
 \usepackage[text={130mm,200mm},centering]{geometry} 

\hypersetup{urlcolor=blue, citecolor=red}

\newcommand{\be}{\begin{equation}\begin{aligned}}
\newcommand{\ee}{\end{aligned}\end{equation}}
\newcommand{\ben}{\begin{equation}\nonumber\begin{aligned}}

 \newcommand{\A}{\mathscr{A}}

 \newcommand{\R}{\mathbb{R}}

\newcommand{\N}{\mathbb{N}}

\renewcommand{\d}{{\rm d}}

\newcommand{\dist}{{\rm dist}}

\renewcommand{\leq}{\leqslant}
\renewcommand{\geq}{\geqslant}

\newtheorem{assumption}{Assumption}

\begin{document}

\title[Asymptotically autonomous pullback attractors]{Convergences of asymptotically autonomous pullback attractors towards  semigroup attractors} 

\author[H. Cui ]{} 


\keywords{Asymptotically autonomous dynamical system; pullback attractor;
      \hfill\break\indent  limiting equation; upper semi-continuity}
   \thanks{{\emph{Email address:} {\color{blue}h.cui@outlook.com} (H. Cui)} \qquad  \qquad  \qquad  \qquad   \qquad  \qquad  \qquad  \qquad  \qquad {\today}}

\maketitle

\centerline{\scshape Hongyong Cui }
\medskip
{\footnotesize
 \centerline{School of Mathematics  and Statistics,}
   \centerline{ Huazhong University of Science \& Technology, }
   \centerline{ Wuhan 430074,  P.R. China}
}
 
\begin{abstract}
For pullback attractors of asymptotically autonomous dynamical systems we   study   the convergences of their   components     towards the global attractors of the limiting semigroups. We use some conditions of uniform boundedness of pullback attractors,  instead of   uniform compactness conditions used in the literature. Both forward convergence and backward convergence are studied.
\end{abstract}

\numberwithin{equation}{section}
\newtheorem{theorem}{Theorem}[section]
\newtheorem{lemma}[theorem]{Lemma}
\newtheorem{proposition}[theorem]{Proposition}
\newtheorem{remark}[theorem]{Remark}
\newtheorem{definition}[theorem]{Definition}
\allowdisplaybreaks

\section{Introduction}

The theory of pullback attractors is a  useful tool  to  study  the long time behavior of evolution systems with non-autonomous forcings. Unlike autonomous systems, long time behavior in non-autonomous systems have interpretations pullback and forward.

 In  dynamical
system theory, such a non-autonomous evolution system is often   formulated as  a \emph{process}, i.e., a mapping $U: \R^2_\geq \times X\to X$ satisfying $U(\tau,\tau,x)=x$    and  $ U(t,s,U(s,\tau,x))=U(t,\tau,x)$   for all $t\geq s\geq \tau$ and $x\in X$, where   $\R^2_{\geq }:=\{(t,\tau)\in \R^2: t\geq \tau\}$ and $ X  $  a complete metric space, while the    pullback attractor $\mathfrak{A}$ of a process $U$  is defined as   a    compact non-autonomous   set in the form $\mathfrak{A}=\{A(t)\}_{t\in\R}$ which is the minimal among those that are invariant and  pullback attract non-empty bounded sets in $X$.   The pullback attractor gives rich information of the asymptotic dynamics of the system from the past, and has  close relationship  to other kinds of attractors,  such as  uniform attractors, cocycle attractors, etc., see \cite{carvalho15dcds,cui17jde,cui17dcds}.  Its time-dependence   is directly related to the non-autonomous  characteristic of the system.

Intuitively, when 
 the non-autonomous forcing of a dynamical system becomes more and more autonomous, i.e., mathematically,  converges in time  to a time-independent forcing in some sense, the  non-autonomous nature of the pullback attractor should  correspondingly become  weaker and weaker. Hence, if the limiting  system is an autonomous semigroup with a global attractor, the pullback attractor should, in some sense, converge to that  global attractor in  time.  This motivates to the  asymptotically autonomous  study of pullback attractors, see, e.g. 
\cite{carvalho13,kloeden15jmaa,kloeden17jmaa,li18jmaa}. 

Most recently, Li et al$.$  \cite{li18jmaa} showed that
\begin{theorem}[\cite{li18jmaa}]\label{li}
Let $\mathfrak{A}$ be  the pullback attractor of a process $U$ and   $\A$  the global attractor of a semigroup $S$. Suppose  that
\begin{itemize}
\item[(i)] 
 for any $\{x_\tau\}$ with $\lim_{\tau\to \infty} d_X(x_\tau,x_0)=0$,
\be \label{li01}
\lim_{\tau\to\infty} \dist\big(U(\tau+t,\tau, x_\tau), S(t,x_0)\big)=0,\quad \forall t\geq 0 ;
\ee 
\item[(ii)]  the pullback attractor $\mathfrak{A}$ is forward compact, i.e.,  the union $\cup_{s\geq 0} A(s)$ is precompact. 
\end{itemize}
Then  $\lim_{t\to\infty} \dist \big(A(t),\A \big)=0$.
\end{theorem}
 The above theorem     improves the corresponding results of Kloeden and Simsen \cite{kloeden15jmaa} who used similar compactness conditions but with more  uniformity.  However,   in   view of  applications,    proving the compactness of the attractor itself is often where the real difficulty lies, especially in cases  in which Sobolev compact embeddings cannot help. 

In this paper we   give an alternative theorem, making use of a \emph{forward  boundedness} condition  instead of the   \emph{forward  compactness} condition and with condition \eqref{li01} slightly modified.  By an   example on an unbounded domain presented in the last section, we show that the modified version of condition   \eqref{li01} can be verified quite easily and  that the forward boundedness condition can be obtained directly by estmiates of solutions, without any particular techniques needed  as previously to show   forward  compactness.

 We also study the backwards  convergence problem, i.e., of   $ A(t) $ converging to $\A$ as $t\to -\infty$.  We first give a theorem using compactness conditions  analogously to Theorem \ref{li}, and then 
 improve it with boundedness conditions.   Remarkably, the use of boundedness conditions  also enables us   to show that  the    backwards convergence can be in the full Hausdorff metric sense, not only in the semi-metric sense. In other words, the global attractor can be the $\alpha$-limit set of the pullback attractor.

 Note that similar topics to asymptotic equivalence have also been studied for skew-product flows and ordinary differential equations, see, e.g., \cite{sell71,kloeden11,wang04na,cui17dcds}. Our results  do not require a compact phase space, and we  obtain  convergences in the full Hausdorff metric sense rather than only in the semi-metric sense. In addition,  no continuity conditions  of the dynamical systems are assumed throughout this paper.

 \section{Sufficient conditions} 
 
 In this section we formulate our theoretical results on sufficient condtions ensuring the   convergences of pullback attractors  towards global attractors. 
 For  a complete metric space $(X, d_X)$ we  denote by $\dist$ the Hausdorff semi-metric between nonempty sets, i.e.,
 \[
  \dist(A,B) :=\sup_{a\in A}\inf_{b\in B}d_X(a,b),
 \]
 and denote  the Hausdorff metric by $\dist_H(A,B):=\max\{\dist(A,B),\dist(B,A)\}$.

   \begin{definition}
   A family $\{E(t)\}_{t\in \R}$ of nonempty sets is said to be 
   \begin{itemize}
    \item[(i)]  \emph{forward bounded/compact}, if there exists a bounded/compact  set $B$ such that $$\bigcup_{t\geq 0} E(t)\subset B ; $$
        \item[(ii)]  \emph{backwards bounded/compact}, if there exists a bounded/compact  set $K$ such that 
        $$\bigcup_{t\leq 0} E(t)\subset K.$$
   \end{itemize}
   \end{definition}

    \subsection{Forward convergence in distant future}
       
       Now we establish an alternative theorem for Theorem \ref{li}, using forward boundedness condition instead of the forward compactness condition.

     \begin{theorem} \label{26.0}
     Suppose that $U$ is a process with  pullback attractor $\mathfrak{A} =\{A(t)\}_{t\in\R}$   and $S$  is a semigroup with  global attractor $\A$.  
    If
    \begin{itemize}
    \item[(i)]  $\mathfrak{A}$ is forward bounded, i.e.,
     there is a bounded set $B $ such that $\cup_{t\geq 0}A(t)\subset B $;
     \item[(ii)]
     the following asymptotically autonomous condition holds
     \be \label{7.1}
    \lim_{t\to\infty} \sup_{x\in 
    B } d_X \big( U(t+T,t,x), S(T ,x)\big) = 0 ,\quad \forall T>0,
     \ee 
     \end{itemize}
     then 
    \be \label{26.5}
     \lim_{t\to \infty}  \dist\big(A(t),\A\big) =0 .
    \ee
     \end{theorem}
     \begin{proof}

     If it is not the case, then there exist $ \delta>0$ and $t_n\to \infty$ such that 
      \ben
   \dist\big(A (  {t_n}), \A \big) > \delta,\quad \forall n\in \N  .
    \ee 
   By the compactness  of the attractor $\mathfrak{A}$, for any $n\in \N$  there exists an  $x_{n }\in A ( {t_n} ) $ such that  
    \be \label{26.8}
    \dist(x_{n },\A)= \dist\big(A ({t_n} ), \A\big)>\delta.
    \ee
      Since $B $ is attracted by $\A$ under $S$,   there exists a $T>0$ such that 
    \be \label{26.7}
   \dist\big(S(T, B ), \A\big) < \delta/2. 
    \ee 
    In addition, by the invariance  of $\mathfrak{A}$, for each $n\in \N$    there exists  $y_{n}\in A({t_n-T} )\subset B  $ such that  $
      x_{n }=U(t_n, t_n-T,  y_{n})$,
 which along with \eqref{26.8} implies 
    \be \label{26.1}
     \dist\big(U(t_n, t_n-T,  y_{n }), \A \big) = \dist(x_{n },\A)>\delta ,\quad \forall n\in\N.
    \ee
    On the other hand, by condition \eqref{7.1}, there exists an  $N=N(\delta)>0$ such that 
    \ben
   &  d_X\big(U(t_N, t_N-T,  y_{N}), 
    S(T,  y_{N})\big)  
     \leq  \sup_{x\in B}d_X\big(U(t_N, t_N - T,  x), 
    S(T, x)\big)
    <\delta /2,  
    \ee 
    so,  by  \eqref{26.7},    
    \ben &
   \dist\big(  U (t_N, t_N - T,  y_{N}), \A\big )\\   &   \quad
   \leq  
       d_X  \big(U(t_N, t_N  - T,  y_{N}), 
   S(T,  y_{N})\big) \! +  \dist \big( S(T,   y_{N }), \A \big  )  
 < \delta  ,
    \ee
    which contradicts  \eqref{26.1}.  Hence   the theorem holds.
     \end{proof}

 \subsection{Backwards convergence in distant past}
 
   Now we turn to the backwards case.
   Firstly, analogously to Theorem \ref{li} we have the following backwards theorem.
   
\begin{theorem}  \label{16.3}
  Let $\mathfrak{A}$ be  the pullback attractor of a process $U$ and   $\A$  the global attractor of a semigroup $S$.
    Suppose  that 
     \begin{itemize}
     \item[(i)]  for any $\{x_t\}$ with $\lim_{t\to -\infty} d_X(x_t,x_0)=0$,
     \be \label{15.1}
 \lim_{t\to -\infty}   d_X\big(U(t,t- T,x_t), S (T, x_0)\big) = 0,\quad \forall  { T\in \R^+};
    \ee 
         \item[(ii)]   $\mathfrak{A}$ is backwards compact.
          \end{itemize}
         Then 
    $$
  \lim_{t\to-\infty}  \dist \big(A(t), \A \big)=0  .
  $$
  
    \end{theorem}
    \begin{proof} 
 We prove  by contradiction.  
      Suppose that for some $\delta>0$ there exists a sequence $ t_n\to\infty$ such that 
    $$ \dist\big(A(-t_n), \A \big)\geq \delta,\quad \forall n\in \N . $$
  Then by the compactness of $\mathfrak{A}$ there exists a  sequence $x_n\in A(-t_n)$ such that 
    \be
    \label{15.3}
    \dist\big(x_n,  \A \big)=  \dist\big(A(-t_n), \A \big)\geq \delta,\quad \forall n\in \N .
    \ee
 
   Since $\mathfrak{A}$ is backwards compact, the set $B:=\overline{\cup_{t\leq 0} A(t)}$ is  compact. By the forward attraction of $\A$ under $S$, there exists a $T_0>0$ such that 
   \be \label{15.6}
    \dist\big(S(T_0,B), \A \big)<\delta/2.
   \ee
    Besides,
     by the invariance of $\mathfrak{A}$, for every $x_n$ there is   $b_{n} \in A(-t_n-T_0)\subset B$ such that 
    $$ x_n = U(-t_n, -t_n-T_0,  b_{n }),$$
   and    $ b_{n} \to b_0 $ as $n\to \infty$ for some   $b_0\in B $. Hence, by  condition \eqref{15.1},  there exists an $N=N(\delta)>0$ such that  
    \be \label{15.5}
       d_X\big(x_N, S(T_0, b_0 )\big) 
      = d_X \big(U(-t_N, -t_N-T_0, b_{N}), S(T_0, b_0 ) \big)    <  \delta/ 2.
    \ee
        Therefore,  from \eqref{15.5} and \eqref{15.6} it follows that,  
    \ben
    \dist\big(x_N, \A \big) &\leq  d_X\big(x_N, S(T_0, b_0)\big)  +
   \dist \big(S(T_0, b_0),  \A \big) <\delta, 
    \ee
    which contradicts \eqref{15.3}.  Hence we have the theorem.
    \end{proof}

     Next, we establish an alternative theorem using different conditions  thanks to which we can  further  obtain   the convergence in the full Hausdorff metric sense, not only in the semi-metric sense. To show this, it is convenient to begin with a more general   convergence, that of pullback attractors to pullback attractors.
    
\begin{proposition}\label{22.2}
    Suppose that $\mathfrak{A}=\{A(t)\}_{t\in\R}$ and $\mathfrak{A}_\infty=\{A_\infty(t)\}_{t\in\R}$ are pullback attractors of  processes $U$ and $U_\infty$, respectively.  If 
    \begin{itemize}
    \item[(i)]
     $\mathfrak{A}$ is backwards bounded, i.e. there is a bounded set $B$  such that $\cup_{t\leq 0} A(t)\subset B$;
     \item[(ii)] the following convergence holds
    \be \label{24.00}
    \sup_{x\in B, \tau\in \R^+}  d_X \big(U(t,t-\tau,x),U_\infty ( t,t-\tau, x)\big)\to 0,\quad\text{as }t\to -\infty,
    \ee
    \end{itemize}
    then 
    \ben
  \lim_{t\to-\infty}  \dist \big(A(t),A_\infty(t)\big)=0.
    \ee
    If, moreover,  $\mathfrak{A}_\infty$ is also backwards bounded in $B$, then the two attractors $\mathfrak{A} $ and $\mathfrak{A}_\infty$ are asymptotically identitical in distant past, i.e.,
    \ben
  \lim_{t\to-\infty}  \dist_H\big(A(t),A_\infty(t)\big)=0.
    \ee
    \end{proposition}
    \begin{proof} 
    We first prove the first part  by contradiction.
    Suppose that for some $\delta>0$ there exists a sequence $ t_n\to\infty$ such that 
    $$ \dist\big(A(-t_n),A_\infty(-t_n)\big)\geq \delta,\quad \forall n\in \N . $$
  Then by the compactness of $\mathfrak{A}$ there exists a  sequence $x_n\in A(-t_n)$ such that 
    \be
    \label{24.22}
    \dist\big(x_n, A_\infty(-t_n)\big)=  \dist\big(A(-t_n),A_\infty(-t_n)\big)\geq \delta,\quad \forall n\in \N .
    \ee
     
    Since $\mathfrak{A}$ is invariant, for every $m,n\in \N$  we have a  $b_{n,m} \in A(-t_n-m)\subset B$ such that 
    $$ x_n= U(-t_n, -t_n-m, b_{n,m}),\quad \forall  n, m\in \N .$$
    Hence, by  condition \eqref{24.00}, there exists an $N=N(\delta)>0$ such that for all   $m\in \N$ 
    \be \label{24.33}
    &
   d_X\big(x_N,U_\infty (-t_N, -t_N-m, b_{N,m})\big) 
       \\
    &\quad =d_X\big(U(-t_N, -t_N-m, b_{N,m}),U_\infty (-t_N, -t_N-m, b_{N,m})\big) \\
    &\quad  \leq    \sup_{x\in B,\tau\in\R^+} \! d_X\big(U(-t_N, -t_N-\tau, x),U_\infty (-t_N, -t_N-\tau,x)\big)  <  \delta/ 2.
    \ee
    In addition,
    since $\{b_{n,m}\}\subset B$ is pullback attracted by $\mathfrak{A}_\infty$ under $U_\infty$, 
    there is an $M=M(N,\delta) >0$ such that 
    \be \label{24.44}
    & \dist \big(U_\infty (-t_N, -t_N-m, b_{N,m}), A_\infty(-t_N)\big)\\
    &\quad \leq
     \dist\big(U_\infty (-t_N, -t_N-m, B), A_\infty(-t_N)\big)<\delta/2 ,
     \quad \forall m\geq M.
    \ee
    Therefore,  from \eqref{24.33} and \eqref{24.44} it follows  that, for all $m\geq M$, 
    \ben
    \dist\big(x_N, A_\infty(-t_N)\big) &\leq  d_X\big(x_N, U_\infty (-t_N, -t_N-m, b_{N,m})\big)  \\
    &\quad +  
    \dist\big( U_\infty (-t_N, -t_N-m, b_{N,m}),  A_\infty(-t_N)\big) <\delta, 
    \ee
    which contradicts \eqref{24.22}. 
    
   In   case of  $\mathfrak{A}_\infty$ being also backwards bounded,
    exchaging the roles played by $\mathfrak{A}$ and $\mathfrak{A}_\infty$ we have     $
  \lim_{t\to-\infty}  \dist\big(A_\infty(t),A(t)\big)=0,
   $ by which the proof is completed. 
    \end{proof}

    As a corollary of Proposition \ref{22.2}, we have 
\begin{theorem} \label{26.00}
    Suppose that $U$ is a process with  pullback attractor $\mathfrak{A}=\{A(t)\}_{t\in\R}$   and that $S$  is a semigroup with  global attractor $\A$.   If
    \begin{itemize}
    \item[(i)]
     $\mathfrak{A}$ is backwards bounded, i.e. there is a bounded set $B$  such that $\cup_{t\leq 0} A(t)\subset B$;
     \item[(ii)] the following backwards asymptotically autonomous condition    holds
    \be  \label{16.4}
    \sup_{x\in B, \tau\in \R^+} d_X \big(U(t,t-\tau,x),S( \tau, x)\big)\to 0,\quad\text{as }t\to -\infty,
    \ee
    \end{itemize}
   then the global attractor  $\A$  is  the $\alpha$-limit set of the pullback attractor $\mathfrak{A} $, i.e.,
    \ben
  \lim_{t\to-\infty}  \dist_H\big(A(t), \A\big)=0.
    \ee
    \end{theorem}
    \begin{proof}
    Define $U_\infty(t,s,x) \! :=\!  S(t-s,x)$ for  $t\geq s$ and $x\in X$, then    $U_\infty$ is a process with pullback attractor $\mathfrak{A}_\infty$ with $A_\infty(t)\equiv \A$. Hence, by Proposition  \ref{22.2} the theorem  follows.   
    \end{proof}
    
    \section{Necessary conditions}
  
  In the previous section we established sufficient  conditions ensuring the convergence of pullback attractors towards global attractors. In this section, we study necessary conditions.  We  first recall a locally uniform compactness.
   
   \begin{definition}(\cite{cui18pd})
   A  family 
    $\mathfrak E=\{E_t\}_{t\in \R}$   of nonempty compact sets is said to be \emph{locally uniformly compact},  if for any bounded interval $I\subset \R$ the union $\cup_{t\in I}E_t $ is precompact.
   \end{definition}
   For the pullback attractor $\mathfrak A$ of a process $U$, the locally uniform compactness   of $\mathfrak A$ is often trivial since the mapping $t\mapsto A(t)$ is often continuous  
   in the full Hausdorff metric sense, provided that $s\to U(s,\tau ,x)$ is continuous, see    
\cite[p31]{kloeden11},  also \cite{li17dcds}.  In the framework of random attractors Cui et al$.$ \cite{cui18pd} studied the case without the continuity in $s$.
   
   The following proposition indicates that, with the locally uniform compactness,    a pullback attractor $\mathfrak A$ forward   converges to  a global attractor $\A$ implies that  $\mathfrak A$  is forward   compact.  Analogously, the backwards convergence and backwards compactness have the same relationship. Hence, roughly, forward and backward compactnesses  of pullback attractors can be necessary conditions of the corresponding convergences towards global attractors.  Forward and backward boundednesses can be necessary conditions as well. 

\begin{proposition}
Suppose that  $\{E_t\}_{t\in \R}$   is    a  locally uniformly compact family of nonempty compact sets in $X$. 
 Then  there is a  nonempty compact  set $E$ such that $\lim\limits_{t\to\infty}\dist( E_t ,E)=0 $ if and only if   $\{E_t\}_{t\in \R}$   is forward compact.
\end{proposition}
\begin{proof}
\emph{Sufficiency}.  The sufficiency is clear taking $E:=\overline{\cup_{t\geq 0}E_t}$ . 

\emph{Necessity.}   
To prove the forward compactness,  for any   a  sequence $\{x_n\}_{n\in\N}\subset \cup_{t\geq 0}E_t$ we need to prove that $\{x_n\}$ has a convergent subsequence.  Since each $E_t$ is compact, without loss of generality we suppose that there exists increasing  sequence $\{t_n\}_{n\in\N} \subset  [0,\infty) $ such that $x_n\in E_{t_n}$ for each  $n\in\N$. Then two possibilities occur. If $t_n\to\infty$, then as $\dist(x_n,E)\leq \dist(E_{t_n},E)\to 0$ and $E$ is compact,  $\{x_n\}$ has indeed a convergent subsequence;
if 
$\sup_{n\in\N}t_n <a<\infty$, then because of $\{x_n\}\subset \cup_{t\in[0,a]}E_t$ being precompact, $\{x_n\}$ has convergent subsequences as well. 
Hence we have the proposition. 
\end{proof}

     \section{An example of a reaction-diffusion equation on unbounded domain}

     Consider the following   non-autonomous reaction-diffusion equation on $\R$
     \be \label{eq}
     \frac{\d u}{\d t} -\triangle u +\lambda u +f(u)=g(x,t),
     \ee
     with initial condition 
     \be
     u(\tau)=u_0,
     \ee
  where $\lambda >0$ and  the nonlinearity $f\in C^1(\R,\R)$ satisfying   
  \be \label{f}
  f(u)u\geq 0,\quad f(0)=0,\quad f'(u)\geq -c, \quad |f'(u)|\leq c(1+|u|^p)
  \ee 
  with $p\geq 0$. The non-autonomous forcing $g\in L^2_{loc}\big(\R,L^2(\R)\big)$ satisfies  the tempered condition
  \be
  \int^0_{-\infty} e^{\lambda s} \|g(s) \|^2\ \d s <\infty.
  \ee   Consider also the autonomous case with the same conditions
    \be \label{eq2}
     \frac{\d u}{\d t} -\triangle u +\lambda u +f(u)=g_0(x),
     \ee
   where $g_0(x)\in L^2(\R)$. 
  
  It is well-known that the systems  \eqref{eq} and \eqref{eq2} have unique solutions, see, e.g.,  \cite{robinson01,wang07jmaa,wang09fmc,zhu16aml}.
   Besides, under some conditions of $f$ and $g$, the proof of the existence of the pullback attractor  is quite standard, see for instance \cite{cui17dcds}.  For brevity we will not pursue the details here, but   assume  it  and  focus on the asymptotic autonomy   of the pullback attractor.
  
\begin{assumption}
Assume that the systems  \eqref{eq} and \eqref{eq2} both have unique solutions,  and have a  pullback attractor $\mathfrak{A}=\{A(t)\}_{t\in\R}$ and a global attractor $\A$  in $L^2(\R)$,  respectively.
\end{assumption}
The  unique existence of solutions implies that the mapping $U(t,\tau,u_0):=u(t,\tau,u_0)$, $t\geq \tau$, corresponding to the solutions $u$ of \eqref{eq} defines a process, and   $S(s,u_{2,0}): =u_2(s,0,u_{2,0})$ $=u_2(s+\tau,\tau,u_{2,0})$, $s\geq 0,\tau\in \R$, corresponding to the solutions $u_2$ of \eqref{eq2} defines a semigroup.

\begin{lemma}
Any  solution $u(t,\tau,u_0)$ of problem \eqref{eq} satisfies 
 \ben
    \|u(t,\tau, u_0)\|^2 \leq e^{\lambda (\tau-t)}\|u_0\|^2 +c \int^t_{\tau} e^{\lambda (s-t)} \|g(s)\|^2\ \d s,\quad \forall t\geq \tau.
    \ee 
\end{lemma}
   \begin{proof} Taking the inner product of \eqref{eq} with $u$ in $L^2(\R)$ we have
    \ben
    \frac 12 \frac{\d}{\d t} \|u\|^2 +\|\nabla u\|^2+\lambda \|u\|^2+\big(f(u),u \big)=\big(g(t),u \big),
    \ee
    which follows that 
    $$
    \frac{\d}{\d t} \|u\|^2  +\lambda \|u\|^2\leq c \|g(t)\|^2 .
$$
    By Gronwall's inequality we have 
    the result.
   \end{proof}
   
    Hence, define $\mathfrak{B}=\{ B(t)\}_{t\in\R}$ with 
   \be\label{B}
   B(t):=\bigg \{u\in L^2(\R): \|u\|^2\leq c \int^t_{-\infty} e^{\lambda (s-t)} \|g(s)\|^2\ \d s+1\bigg\},\quad \forall t\in \R .
   \ee
   Then $\mathfrak{B}$ is a pullback absorbing set containing the pullback attractor, i.e., $A(t)\subset B(t)$. 
 
 Now we study the forward and backward convergences of the pullback attractor $\mathfrak A$ towards the global attractor $\A$. Naturally, different conditions of the non-autonomous forcing $g(x,t)$ will be required, but we do not need particular techniques to obtain further compactness of the pullback attractor since  our theorems  only need boundedness conditions.  
 
   \begin{theorem}[Forward convergence] 
   \label{th1}
If
 \be \label{cond_g}
\lim_{\tau\to\infty} \int^\infty_\tau \|g(s)-g_0\|^2\ \d s =0 ,
\ee
 then
   \ben
   \lim_{t\to\infty} \dist \big(A(t),\A\big)=0.
   \ee
   \end{theorem} 
   \begin{proof} 
    Let  $B$ be any nonempty bounded set, and let $u_1(t,\tau,x)$ and $u_2(t,\tau,x) $ be the  solutions of \eqref{eq} and \eqref{eq2}, respectively, with the same initial value $x_0 \in B $ at $\tau$. Then the difference $w(t,\tau,0):=u_1(t,\tau,x)-u_2(t,\tau,x)$  satisfies 
    \be \label{31.1}
     \frac{\d w}{\d t} -\triangle w +\lambda w +f(u_1)-f(u_2)=g(t)-g_0,
     \ee
     with initial value $w(\tau)=0$. Taking the inner product with $w$ in $L^2(\R)$ we have 
         \ben
    \frac 12 \frac{\d}{\d t} \|w\|^2 +\|\nabla w\|^2+\lambda \|w\|^2+\big(f(u_1)-f(u_2),w\big)=\big(g(t)-g_0,w\big) .
    \ee
   Hence, by \eqref{f} and Young's inequality,
       \be \label{31.3}
    \frac{\d}{\d t} \|w\|^2 \leq c\|w\|^2+ \|g(t)-g_0\|^2.
    \ee
    Notice that  condition \eqref{cond_g} implies 
    \be \label{g1}
   \lim_{t\to\infty} \int^t_{t-T}  \|g(s)-g_0\|^2\ \d s =0,\quad \forall T>0,
   \ee
 so by Gronwall's inequality we have, for any $T>0$,
    \ben
    \|w(t,t-T,  0)\|^2 & \leq  c \int^t_{t-T} e^{c(t-s)} \|g(s)-g_0\|^2\ \d s \\
    & \leq  c e^{cT}\int^t_{t-T}  \|g(s)-g_0\|^2\ \d s \to 0 , \quad \text{as } t\to \infty,
    \ee 
    i.e.,  condition (ii) in Theorem \ref{26.0} is satisfied.
    
    In the following, by Theorem \ref{26.0},  we  only need to show the forward boundedness of the pullback attractor. 
  By condition  \eqref{cond_g}, there exists an 
    $N\in \N$ such that
    \be
     \int^\infty_N \|g(s)-g_0\|^2\ \d s \leq 1.
    \ee
    Therefore,
     \ben
  & \sup_{t\geq N}  \int^t_{-\infty} e^{\lambda (s-t)} \|g(s)\|^2\ \d s 
   \leq  \sup_{t\geq N}  \int^t_{-\infty} e^{\lambda (s-t)}  \Big(\|g(s)-g_0\|^2 + \|g_0\|^2 \Big) \d s \\
    &\quad \leq  \sup_{t\geq N}  \int^t_{-\infty} e^{\lambda (s-t)} \|g(s)-g_0\|^2\ \d s  +\frac{ \|g_0\|^2}{\lambda } \\
    &\quad \leq  \sup_{t\geq N} \bigg( \int^N_{-\infty} e^{\lambda (s-t)} \|g(s)-g_0\|^2\ \d s +\int^t_N \|g(s)-g_0\|^2\ \d s\bigg) \!+\frac{ \|g_0\|^2}{\lambda }\\
      &\quad \leq    \int^N_{-\infty} e^{\lambda s} \|g(s)-g_0\|^2\ \d s + 1  +\frac{ \|g_0\|^2}{\lambda }<\infty.
    \ee
 This implies that the pullback absorbing set $\mathfrak{B}$ given by \eqref{B} is forward bounded, and so is the pullback attractor.    
   \end{proof}

   \begin{theorem}[Backwards convergence]
   \label{th2}
    If  \be\label{cond_g2}
    \lim_{\tau\to-\infty} \int^\tau_{-\infty} \|g(s)-g_0\|^2\ \d s =0 ,
   \ee
 then the global attractor $\A$ is the $\alpha$-limit set of the pullback attractor $\mathfrak A$, i.e., 
   \ben
   \lim_{t\to-\infty} \dist_H\big(A(t),\A\big)=0.
   \ee
   \end{theorem}
   \begin{proof}
    Let  $B$ be any nonempty bounded set, and let $u_1(t,\tau,x)$ and $u_2(t,\tau,x) $ be the  solutions of \eqref{eq} and \eqref{eq2}, respectively, with the same initial value $x \in B $ at $\tau$. Then the difference $w(t,\tau,0):=u_1(t,\tau,x)-u_2(t,\tau,x)$   satisfies  \eqref{31.3}, i.e.,
    \ben 
    \frac{\d}{\d t} \|w\|^2 \leq c\|w\|^2+ \|g(t)-g_0\|^2.
    \ee
    Since the initial value of  $w $ is always zero,   by Gronwall's inequality we have 
    \ben
    \|w(t,t-T,0)\|^2& \leq \int^t_{t-T} e^{-c(t-s)}\|g(s)-g_0\|^2\ \d s\\
    & \leq \int^t_{-\infty}  \|g(s)-g_0\|^2\ \d s,\quad  \forall T>0.
    \ee
    Thus,  $\sup_{T>0}\|w(t,t-T,0)\|\to0$ as ${t\to -\infty }$, i.e., condition (ii) of Theorem \ref{26.00} holds. 
     
    Next, we need to establish the backward boundedness of $\mathfrak{A}$, i.e.,  condition (i) of Theorem \ref{26.00}.  Since by \eqref{cond_g2} there exists an $M\in \N$ such that 
    \ben
     \int^{-M}_{-\infty} \|g(s)-g_0\|^2  \ \d s  \leq  1,
    \ee
    we have 
    \ben
  & \sup_{t<-M}  \int^t_{-\infty} e^{\lambda (s-t)} \|g(s)\|^2\ \d s 
   \leq \sup_{t\leq -M}  \int^t_{-\infty} e^{\lambda (s-t)} \Big(\|g(s)-g_0\|^2+ \|g_0\|^2\Big) \d s \\
    &\qquad   \leq \sup_{t\leq -M}  \int^t_{-\infty} e^{\lambda (s-t)}  \|g(s)-g_0\|^2 \ \d s +   \sup_{t\leq -M}  \int^t_{-\infty} e^{\lambda (s-t)}   \|g_0\|^2\ \d s  \\
    &\qquad   \leq   \int^{-M}_{-\infty}  \|g(s)-g_0\|^2 \  \d s +   \frac 1\lambda \int^0_{-\infty} e^{\lambda s}   \|g_0\|^2\ \d s <\infty  .
   \ee
   Hence, the pullback absorbing set $\mathfrak{B}$ given by \eqref{B} is backward bounded, and so is the pullback attractor.
   \end{proof}
   
 \begin{remark}
 Slightly modifing the proofs, the conditions \eqref{cond_g} and \eqref{cond_g2} in Theorems \ref{th1} and \ref{th2} can be replaced, respectively,  with 
  \begin{gather} \nonumber
    \lim_{t\to\infty}\|g(t)-g_0\|=0   \quad \text{and } \ \
      \lim_{t\to -\infty}\|g(t)-g_0\|=0  .
   \end{gather}
 \end{remark}
  
     \section{Endnotes}
     In this paper we studied an asymptotically autonomous problem of non-autonomous dynamical systems in terms of  the convergences of the  pullback attractor    towards  the global attractor of the limiting  semigroup. By boundedness conditions we established criteria  Theorems \ref{26.0} and \ref{26.00},  which  weakened the corresponding compactness conditions in Theorems \ref{li} and  \ref{16.3}, 
   with  asymptotic autonomy conditions \eqref{li01} and \eqref{15.1} also modified to conditions \eqref{7.1} and \eqref{16.4}, respectively.  In addition,  as an advantage of \eqref{16.4} over \eqref{15.1}, in Theorem \ref{26.00} we obtained backward convergence in the full Hausdorff metric sense.  All the analysis in this paper does not require any   continuity conditions of the dynamical systems.
   
   In view of applications, Theorems \ref{26.0} and \ref{26.00} can be much easier to apply than  Theorems \ref{li} and \ref{16.3}. First,    boundedness conditions are clearly much easier to establish than compactness conditions; second,  conditions  \eqref{7.1} and \eqref{16.4} in Theorems \ref{26.0} and \ref{26.00} can also be easier to verify than conditions  \eqref{li01} and \eqref{15.1}  in  Theorems \ref{li} and \ref{16.3}, since in \eqref{li01} and \eqref{15.1}   the   process and the limiting semigroup are considered for different initial data which, however, can sometimes cause big difficulties, see \cite[Section 5.2]{cui16svva} for a related  example of multi-valued inclusions.

 \section*{Acknowledgements}  
 On the occasion of Professor Peter Kloeden's 70th birthday,  the author would like to express his sincere thanks  to Professor Peter for his  enlightening advice and patient guidence   during his  stay in Wuhan,  from which  the author learnt a lot  on non-autonomous dynamical systems.    

\end{document}